\documentclass{article}
\usepackage{amssymb}
\usepackage{amsmath}
\usepackage[all]{xy}
\usepackage{hyperref}

\DeclareMathOperator{\Hom}{Hom}

\DeclareMathOperator{\Ker}{Ker}
\DeclareMathOperator{\Ima}{Im}
\DeclareMathOperator{\rad}{rad}

\DeclareMathOperator{\Mod}{Mod}

\DeclareMathOperator{\ssmod}{SSMod}
\DeclareMathOperator{\Soc}{Soc}

\DeclareMathOperator{\idom}{\mathfrak{In}}
\DeclareMathOperator{\pdom}{\mathfrak{Pr}}

\newcommand{\modR}{\Mod\text{-}R}
\DeclareMathOperator{\baer}{\mathbb{Z}^{\mathbb{N}}}

\DeclareMathOperator{\cat}{\mathbb{Z}^{(\mathbb{N})}}
\DeclareMathOperator{\Z}{\mathbb{Z}}

\newcommand{\subpr}[1]{\underline{\pdom}^{-1}(#1)}
\newcommand{\zhom}[2]{\Hom(#1,#2)}
\newcommand{\subinj}[1]{\underline{\idom}^{-1}(#1)}

\newtheorem{theorem}{Theorem}[section]
\newtheorem{prop}[theorem]{Proposition}
\newtheorem{definition}[theorem]{Definition}
\newtheorem{cor}[theorem]{Corollary}
\newtheorem{lemma}[theorem]{Lemma}

\newtheorem{example}[theorem]{Example}

\newenvironment{proof}[1][\it{Proof}]{\textbf{#1. } }{$\square$}

\begin{document}

\title{An alternative perspective on projectivity of modules}


\author{Chris Holston \\ \small{holston@ohio.edu} \\
 Sergio R. L\'opez-Permouth \\ \small{lopez@ohio.edu} \\
  Joseph Mastromatteo \\ \small{jm424809@ohio.edu} \\
   Jos\'e E. Simental-Rodr\'{i}guez\footnote{Current address: Department of Mathematics. Northeastern University. Boston, MA 02115. simentalrodriguez.j@husky.neu.edu} \\ \small{js224910@ohio.edu} \\
Department of Mathematics.\\
Ohio University. Athens, OH. 45701 USA.\vspace{10pt} \\
\it{Dedicated to the memory of our dear colleague and friend Francisco Raggi.}}





\date{\today}

\maketitle

\begin{abstract}
We approach the analysis of the extent of the projectivity of modules from a fresh perspective as we introduce the notion of relative subprojectivity. A module $M$ and is said to be {\em $N$-subprojective} if for every epimorphism $g:B \rightarrow N$ and homomorphism $f:M \rightarrow N$, there exists a homomorphism $h:M \rightarrow B$ such that $gh=f$. For a module $M$, the {\em subprojectivity domain of $M$} is defined to be the collection of all modules $N$ such that $M$ is $N$-subprojective. We consider, for every ring $R$, the subprojective profile of $R$, namely, the class of all subprojectivity domains for $R$ modules. We show that the subprojective profile of $R$ is a semilattice, and consider when this structure has coatoms or a smallest element. Modules whose subprojectivity domain is smallest as possible will be called {\em subprojectively poor} ({\em $sp$-poor}) or {\em projectively indigent} ({\em p-indigent}) and those with co-atomic subprojectivy domain are said to be {\em maximally subprojective}. While we do not know if $sp$-poor modules and maximally subprojective modules exist over every ring, their existence is determined for various families. For example, we determine that artinian serial rings have $sp$-poor modules and attain the existence of maximally subprojective modules over the integers and for arbitrary V-rings. This work is a natural continuation to recent papers that have embraced the systematic study of the injective, projective and subinjective profiles of rings. \\

\end{abstract}


\section{Introduction and Preliminaries}

The purpose of this paper is to initiate the study of an alternative perspective on the analysis of the projectivity of a module, as we introduce the notions of relative subprojectivity and assign to every module its subprojectivity domain.  A module is projective if and only if its subprojectivity domain consists of all modules. Therefore, at this extreme, there is no difference in the role played by the projectivity and subprojectivity domains.  Interesting things arise, however, when we focus on the subprojectivity domain of modules which are not projective.  It is easy to see that every module is subprojective relative to all projective modules, and one can show (Proposition \ref{basicfact2}) that projective modules are the only ones sharing the distinction of being in every single subprojectivity domain.  It is thus tempting to ponder the existence of modules whose subprojectivity domain consists precisely of only projective modules.  We refer to these modules as {\em $sp$-poor} or, to keep in line with \cite{pinar11}, we sometimes use the expression {\em p-indigent} .

This paper is inspired by similar ideas and notions studied in several papers. On the one hand, relative injectivity, injectivity domains and the notion of a {\it poor} module (modules with smallest possible injectivity domain) have been studied in \cite{lopez10}, \cite{lopez11} and \cite{lopez12}. Dually, relative projectivity, projectivity domains and the notion of a {\it p-poor} module have been studied in \cite{holston11} and \cite{lopez12}. On the other hand, in \cite{lopez12} the authors name a class of modules an $i$-portfolio (resp. $p$-portfolio) if it coincides with the injectivity (resp. projectivity) domain of some module. Then, they proceed to define the injective profile (resp. projective profile) of a ring $R$, an ordered structure consisting on all the $i$-portfolios (resp. $p$-portfolios) in $\modR$. In this paper, we study these concepts in the context of subinjectivity and subprojectivity domains, thus obtaining the ordered invariants $\mathfrak{siP}(R)$ and $\mathfrak{spP}(R)$, the subinjective and subprojective profile of $R$, respectively. We study some of its properties, such as the existence of coatoms and their relations with the lattice of torsion theories in $\modR$.

One of the first things that comes to the surface in this type of study is the potential existence of modules which are least injective or projective possible with respect to whichever measuring approach one may  be using. Injectively and projectively poor modules have been studied in \cite{lopez10}, \cite{lopez11}, \cite{lopez12} and \cite{holston11}. In \cite{pinar11}, Aydo\u{g}du and L\'opez-Permouth modify in a subtle yet significant way the notion of relative injectivity to obtain {\em relative subinjectivity}. They also study subinjectivity analogs of poor modules, calling them {\em indigent}. Here, we study the projective analog of relative subinjectivity and indigent modules. In order to emphasize the analogy between poor and indigent modules, we also call indigent modules {\em subinjectively poor}, or $si$-poor for short. In this same spirit, we call the subprojective analog of indigent modules either $p$-indigent or $sp$-poor.

We depict the different analogies between the different ways of measuring the injectivity and projectivity of a module in the following diagram.

$$
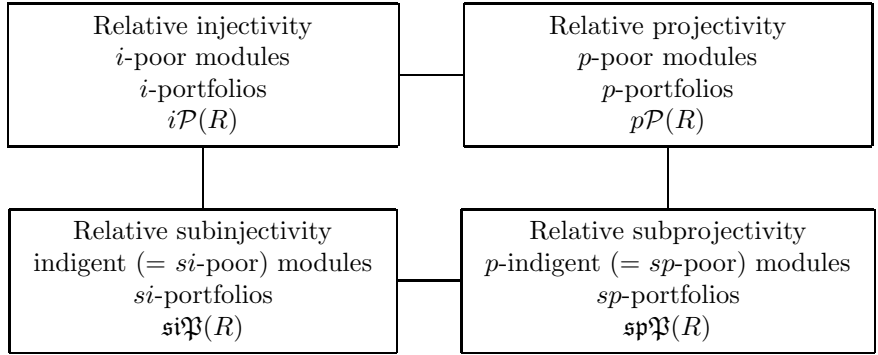
\xymatrix{*+[F]{\begin{tabular}{c} \, \; \; \; \text{Relative injectivity} \, \; \; \;  \\ \text{$i$-poor modules} \\ \text{$i$-portfolios} \\ $i\mathcal{P}(R)$ \end{tabular}} \ar@{-}[r] \ar@{-}[d] & *+[F]{\begin{tabular}{c} \, \; \; \; \text{Relative projectivity} \; \; \; \, \\ \text{$p$-poor modules} \\ \text{$p$-portfolios} \\ $p\mathcal{P}(R)$\end{tabular}} \ar@{-}[d] \\
*+[F]{\begin{tabular}{c}\text{Relative subinjectivity} \\ \text{indigent (= $si$-poor) modules} \\ \text{$si$-portfolios} \\ $\mathfrak{siP}(R)$ \end{tabular}} \ar@{-}[r] & *+[F]{\begin{tabular}{c} \text{Relative subprojectivity} \\ \text{$p$-indigent (= $sp$-poor) modules} \\ \text{$sp$-portfolios} \\ $\mathfrak{spP}(R)$ \end{tabular} }} 
$$


Before tackling the rest of the paper, we finish the section with a review of some of the needed background material.
Throughout, $R$ will denote an associative ring with identity and modules will be unital right $R$-modules, unless otherwise explicitely stated.  As usual, we denote by $\modR$ the category of right $R$-modules.  If $M$ is an $R$-module, then $\rad(M)$, $\Soc(M)$ and pr.dim$(M)$ will respectively denote the Jacobson radical, socle and projective dimension of $M$.  The Jacobson radical of a ring $R$ will be denoted by $J(R)$.  A ring $R$ is called a {\em right $V$-ring} if every simple $R$-module is injective; a {\em right hereditary ring} if submodules of projective modules are projective or, equivalently, if quotients of injective modules are injective; a {\em right perfect ring} if every module has a projective cover; a {\em semiprimary ring} if $J(R)$ is nilpotent and $R/J(R)$ is a semisimple artinian ring; and a {\em right coherent ring} if every finitely generated right ideal is finitely presented or, equivalently, if products of flat left $R$-modules are flat.

In \cite{lopez12} torsion theory is used as a tool in the study of relative injectivity and projectivity. Such notions are also employed here so, for easy reference, we recall them now. A torsion theory $\mathbb{T}$ is a pair of classes of modules $(\mathcal{T},\mathcal{F})$ such that (i) $\zhom{M}{N} = 0$ for every $M \in \mathcal{T}$, $N \in \mathcal{F}$; (ii) if $\zhom{A}{N} = 0$ for all $N \in \mathcal{F}$, then $A \in \mathcal{T}$; and (iii) if $\zhom{M}{B} = 0$ for all $M \in \mathcal{T}$, then $B \in \mathcal{F}$. In this situation, $\mathcal{T}$ and $\mathcal{F}$ are called the torsion class and torsion-free class of $\mathbb{T}$, respectively. A class of modules $\mathcal{T}$ is the torsion class of some torsion theory if and only if it is closed under quotients, extensions and arbitrary direct sums. Likewise, a class of modules $\mathcal{F}$ is the torsion-free class of some torsion theory if and only if it is closed under submodules, extensions and arbitrary direct products. If $\mathcal{C}$ is a class of modules, then $\mathbb{T}_{\mathcal{C}} := (\mathcal{T}_{\mathcal{C}}, \mathcal{F}_{\mathcal{C}})$, where $\mathcal{F}_{\mathcal{C}} = \{N \in \modR : \zhom{C}{N} = 0 \; \text{for every} \; C \in \mathcal{C}\}$ and $\mathcal{T}_{\mathcal{C}} = \{M \in \modR : \zhom{M}{N} = 0 \; \text{for every} \; N \in \mathcal{F}_{\mathcal{C}}\}$ is said to be the torsion theory generated by $\mathcal{C}$. Likewise, $\mathbb{T}^{\mathcal{C}} = (\mathcal{T}^{\mathcal{C}}, \mathcal{F}^{\mathcal{C}})$ where $\mathcal{T}^{\mathcal{C}} = \{M \in \modR : \zhom{M}{C} = 0 \; \text{for every} \; C \in \mathcal{C}\}$ and $\mathcal{F}^{\mathcal{C}} = \{N \in \modR : \zhom{M}{N} = 0  \; \text{for every} \; M \in \mathcal{T}^{\mathcal{C}}\}$ is said to be the torsion theory cogenerated by $\mathcal{C}$, \cite[Chapter VI]{stenstrom}. The torsion theory $\mathbb{T}_{\mathcal{C}}$ (resp. $\mathbb{T}^{\mathcal{C}})$ can also be characterized as the smallest torsion theory such that every object in $\mathcal{C}$ is torsion (resp. torsion-free). If $M \in \modR$, we write $\mathbb{T}_M$ and $\mathbb{T}^M$ for $\mathbb{T}_{\{M\}}$ and $\mathbb{T}^{\{M\}}$, respectively.

Recall that a module $M$ is said to be quasi-projective if it is projective relative to itself. Over a right perfect ring $R$, every quasi-projective module $M$ satisfies the following conditions:

\begin{tabular}{rl}
(D1) & For every submodule $A$ of $M$, there is a decomposition $M = M_1 \oplus M_2$ \\ &  such that $M_1 \leq A$ and $A \cap M_2 \ll M$. \\
(D2) & If $A \leq M$ is such that $M/A$ is isomorphic to a direct summand of $M$, \\ &  then $A$ is a direct summand of $M$. \\
(D3) &  If $M_1$ and $M_2$ are direct summands of $M$ with $M_1 + M_2 = M$, then \\ &  $M_1\cap M_2$ is a direct summand of $M$.
\end{tabular}

Modules satisfying (D1) are called lifting, see \cite{lifting}. Modules satisfying (D1) and (D2) are called discrete, while modules satisfying (D1) and (D3) are called quasi-discrete. Every discrete module is quasi-discrete, as it is the case that (D2) $\Rightarrow$ (D3), \cite[Lemma 4.6]{mohamed90}. It is not the case that every projective module is lifting, as, for example $\mathbb{Z}$ is not a lifting $\mathbb{Z}$-module. However, if $R$ is right perfect, then every projective module is discrete, cf. \cite[Theorem 4.41]{mohamed90}. Every quasi-discrete module decomposes as a direct sum of modules whose every submodule is superfluous, see \cite[Theorem 4.15]{mohamed90}. 

For additional concepts and results not mentioned here, we refer the reader to \cite{anderson92}. \cite{assem} and \cite{lam99}.

\section{Subprojectivity and the subprojectivity domain of a module}

\begin{definition}
Given modules $M$ and $N$, $M$ is said to be {\em $N$-subprojective} if for every epimorphism $g:B\rightarrow N$ and for every homomorphism $f:M\rightarrow N$, then there exists a homomorphism $h:M\rightarrow B$ such that $gh=f$.  The {\em subprojectivity domain}, or {\em domain of subprojectivity}, of a module $M$ is defined to be the collection
\[\subpr{M} \, := \{ \, N\in \modR \, : \, M \text{ is } N\text{-subprojective} \, \}.\]
\end{definition}

The domain of subprojectivity of a module is a measure of how projective that module is.  Just as with projectivity domains, a module $M$ is projective precisely when $\subpr{M}$ is as large as possible (i.e. equal to $\modR$.) 

Before we proceed, we need to introduce two additional notions.

\begin{definition}
Let $\mathcal{C} \subseteq \modR$. We say that $\mathcal{C}$ is a {\em subprojective-portfolio}, or {\em $sp$-portfolio} for short, if there exists $M \in \modR$ such that $\mathcal{C} = \subpr{M}$.
The class $ \mathfrak{spP}(R) := \{\mathcal{C} \subseteq \modR : \mathcal{C} \; \text{is an} \; sp\text{-portfolio}\}$ will be named the {\em subprojective profile}, or {\em $sp$-profile}, of $R$.
\end{definition}

Our first lemma says that, in order for $M$ to be $N$-subprojective, one only needs to lift maps to projective modules that cover $N$, to free modules that cover $N$ or even to a single projective module that covers $N$.

\begin{lemma}\label{onlyfrees}
Let $M, N \in \modR$. Then, the following conditions are equivalent.
\begin{enumerate}
\item $M$ is $N$-subprojective.
\item For every $f: M \rightarrow N$ and every epimorphism $g: P \rightarrow N$ with $P$ projective, there exists $h: M \rightarrow P$ such that $gh = f$. 
\item For every $f: M \rightarrow N$ and every epimorphism $g: F \rightarrow N$ with $F$ free, there exists $h: M \rightarrow F$ such that $gh = f$.
\item For every $f: M \rightarrow N$ there exists an epimorphism $g: P \rightarrow N$ with $P$ projective and a morphism $h: M \rightarrow P$ such that $gh = f$.
\end{enumerate}
\end{lemma}
\begin{proof}
The implications $(1) \Rightarrow (2) \Rightarrow (3) \Rightarrow (4)$ are clear. To show $(4) \Rightarrow (1)$, assume $(4)$ and let $f: M \rightarrow N$ be a morphism and $\overline{g}: B \rightarrow N$ be an epimorphism. By $(4)$, there exist an epimorphism $g: P \rightarrow N$ and a morphism $h: M \rightarrow P$ such that $gh = f$. Since $P$ is projective, there exists a morphism $\overline{h}: P \rightarrow B$ such that $g = \overline{g}\overline{h}$. Then, $\overline{h}h: M \rightarrow B$ and $\overline{g}\overline{h}h = gh = f$. Hence, $M$ is $N$-subprojective.
\end{proof}

Using the preceding lemma, we can show that a module $M$ is projective if and only if it is $M$-subprojective, thus ruling out the possibility of a non-trivial subprojective analogue to the notion of quasi-projectivity.

\begin{prop}\label{basicfact1}
For any module $M$, the following are equivalent;

  \begin{enumerate}
    \item $M$ is projective.
    \item $M \in \subpr{M}$.
  \end{enumerate}
\end{prop}

\begin{proof}
The implication $(1)\Rightarrow (2)$ is clear. For $(2)\Rightarrow (1)$, put $M = N$ and $f = 1_M$, the identity morphism on $M$ in the condition of Lemma \ref{onlyfrees} (4), to see that $M$ is a direct summand of a projective module. Hence, $M$ is projective.
\end{proof}

Some modules can be shown easily to belong to a subprojectivity domain.

\begin{prop}\label{hom=0}
If $\Hom_R(M,A)=0$, then $A \in \subpr{M}$.
\end{prop}

\begin{proof}
If $\Hom_R(M,A)=0$, then given any epimorphism $g:C \rightarrow A$ if we let $h:M \rightarrow C$ be the zero mapping then $gh=0$.  Whence $A \in \subpr{M}$.
\end{proof}

As an easy consequence of Proposition \ref{hom=0}, we have the following.

\begin{cor}\label{cor to hom=0}
Let $M$ and $A$ be right $R$-modules. Then,
	\begin{enumerate}
		\item If $\rad(M)=M$ and $\rad(A)=0$, then $M$ is $A$-subprojective.
		\item If $M$ is singular and $A$ is nonsingular, then $M$ is $A$-subprojective.
		\item If $M$ is semisimple and $\Soc(A) = 0$, then $M$ is $A$-subprojective.
	\end{enumerate}
\end{cor}
Proposition \ref{hom=0} is also instrumental in figuring out the next example, where we see that sometimes the conditioned study in that Proposition actually characterizes certain 
subprojectivity domains.  This subject will be picked up again later in Proposition \ref{basic}.

\begin{example}\label{rationals}
We see that in the category of $\mathbb{Z}$-modules, $\subpr{\mathbb{Q}}$ consists precisely of the abelian groups granted by Proposition \ref{hom=0} for, if there is a nonzero morphism $f: \mathbb{Q} \rightarrow M$, let $\pi: F \rightarrow M$ be an epimorphism with $F$ free. Since there are no nonzero morphisms from $\mathbb{Q}$ to $F$, we cannot lift $f$ to a morphism $\mathbb{Q} \rightarrow F$, so $M \not\in \subpr{\mathbb{Q}}$. Consequently, the subprojectivity domain of $\mathbb{Q}$ consists precisely of the class of reduced abelian groups. Note that a similar technique can be used to find the subprojectivity domain of any divisible abelian group. We further explore this phenomena in Section 3 of this paper. 
\end{example}

 It is a natural question to ask how small $\subpr{M}$ can be.  The next proposition shows that the domain of subprojectivity of any module must contain at least the projective modules, and the projective modules are the only ones that belong to all $sp$-portfolios.

\begin{prop}\label{basicfact2}
The intersection $\bigcap\subpr{M}$, running over all $R$-modules $M$, is precisely $\{P \in \modR \,|\, P \text{ is projective} \}$.
\end{prop}

\begin{proof}
To show the containment $\subseteq$, suppose $M$ is a module which is subprojective relative to all $R$-modules. Then in particular, $M \in \subpr{M}$. So by Proposition \ref{basicfact1}, $M$ is projective.

To show the containment $\supseteq$, let $P$ be a projective module and $M$ be any $R$-module. Let $g:B \rightarrow P$ be an epimorphism and $f:M \rightarrow P$ be a homomorphism. Now since $P$ is projective $g$ splits and so there exists a homomorphism $k:P \rightarrow B$ such that $gk=1_P$. Then $g(kf)=(gk)f=f$ and so by definition $P \in \subpr{M}$.  Since $M$ was arbitrary the result follows.
\end{proof}

Proposition \ref{basicfact2} provides a lower bound on how small the domain of subprojectivity of a module can be.  If a module does achieve this lower bound, then we will call it suprojectively-poor.

\begin{definition}
A module $M$ is called {\em subprojectively poor}, {\em $sp$-poor}, or {\em $p$-indigent}, if its subprojectivity domain consists of only the projective modules.
\end{definition}

Notice that it is not clear whether $sp$-poor modules over a ring $R$ must exist.  Section 4 will be devoted to this problem, but first we go deeper into our study of subprojectivity.

The following several propositions show that subprojectivity domains behave nicely with respect to direct sums.

\begin{prop}\label{subprojectivity domains of sums}
Let $\{M_i\}_{i \in I}$ be a set of $R$-modules. Then, $\subpr{\bigoplus_{i \in I} M_i}=\bigcap_{i \in I} \subpr{M_i}$, that is, the subprojectivity domain of a direct sum is the intersection of the subprojectivity domains of the summands.
\end{prop}

\begin{proof}
To show the containment $\subseteq$, let $N$ be in the subprojectivity domain of $\bigoplus_{i \in I} M_i$ and fix $j \in I$.  Let $g:B\rightarrow N$ be an epimorphism and $f:M\rightarrow N$ be a homomorphism.  Let $p_j:\bigoplus_{i \in I} M_i \rightarrow M_j$ denote the projection map and $e_j:M_j \rightarrow \bigoplus_{i \in I} M_i$ denote the inclusion map. Since $N \in \subpr{\bigoplus_{i \in I} M_i}$, then there exists a homomorphism $h^\prime :\bigoplus_{i \in I} M_i \rightarrow B$ such that $gh^\prime =fp_j$.  Letting $h := h^\prime e_j:M_j \rightarrow B$, then it is straightforward to check that $gh=f$.  Whence, $N \in \subpr{M_j}$.

To show the containment $\supseteq$, let $N$ be in the subprojectivity domain of $M_i$ for every $i \in I$, let $g:B \rightarrow N$ be an epimorphism and let $f:\bigoplus_{i \in I}M_i \rightarrow N$ be a homomorphism.  Since for each $j \in I$, $N \in \subpr{M_j}$ then $\exists h_j:M_j \rightarrow B$ such that $fe_j=gh_j$.  Letting $h := \bigoplus_{i \in I} h_i:\bigoplus_{i \in I} M_i \rightarrow B$, then 

\[gh=\bigoplus_{i \in I} gh_i=\bigoplus_{i \in I} fe_i=f\bigoplus_{i \in I} e_i=f.\]

Hence, $N \in \subpr{\oplus_{i\in I} M_i}$.
\end{proof}

Note that Proposition \ref{subprojectivity domains of sums} tells us that $\mathfrak{spP}(R)$ is a semilattice with a biggest element, namely $\modR$, the subprojectivity domain of any projective $R$-module. In view of Proposition \ref{basicfact2}, $\mathfrak{spP}(R)$ has a smallest element if and only if $R$ has an $sp$-poor module.
\begin{prop}\label{sums in subprojectivity domain}
If $N \in \subpr{M}$, then every direct summand of $N$ is in $\subpr{M}$.
\end{prop}

\begin{proof}
Suppose $A$ is a direct summand of $N$, and let $g:C \rightarrow A$ be an epimorphism and $f:M \rightarrow A$ be a homomorphism.  Consider the epimorphism $g \oplus 1:C \oplus N/A \rightarrow A \oplus N/A \cong N$, where $1:N/A \rightarrow N/A$ is the identity map.  Since $N \in \subpr{M}$, then there exists a homomorphism $\hat{h}:M \rightarrow C \oplus N/A$ such that $(g \oplus 1)\hat{h}=ef$, where $e:A \rightarrow N$ is the inclusion map.  Therefore,
\[g(p\hat{h})=p(g \oplus 1)\hat{h}=p(ef)=f,\]
where $p:N \rightarrow A$ denotes the projection map.  Hence, $A \in \subpr{M}$.
\end{proof}

\begin{prop}\label{closure of finite direct sums}
If $A_i\in \subpr{M}$ for $i\in \{1,\cdots,m\}$, then $\bigoplus_{i=1}^m A_i \in \subpr{M}$.
\end{prop}

\begin{proof}
By induction, it is sufficient to prove the proposition when $m=2$.  Let $g:C \rightarrow A \oplus B$ be an epimorphism and $f:M \rightarrow A \oplus B$ be a homomorphism.  Since $A\in \subpr{M}$, then there exists a homomorphism $h_1:M \rightarrow C$ such that $p_Agh_1=p_Af$, where $p_A:A \oplus B \rightarrow A$ is the projection map.  So $p_A(gh_1-f)=0$ and hence $\Ima(gh_1-f) \subset 0 \oplus B \cong B$.  Since $B \in \subpr{M}$, then there exists a homomorhism $h_2:M \rightarrow g^{-1}(0 \oplus B) \subset C$ such that $gh_2=gh_1-f$.  Let $h:=h_1-h_2$.  Then
\[gh=gh_1-gh_2=gh_1-gh_2-f+f=gh_2-gh_2+f=f.\]
Hence, $A \oplus B \in \subpr{M}$.
\end{proof}

\begin{prop}
If $M$ is finitely generated and $A_i$-subprojective for every $i\in I$, then $M$ is $\bigoplus_{i\in I} A_i$-subprojective.
\end{prop}

\begin{proof}
Let $f:M\rightarrow \bigoplus_{i\in I} A_i$ be a homomorphism and $g:C\rightarrow \bigoplus_{i\in I} A_i$ be an epimorphism.  Let $X:=\{m_1,\ldots,m_k\}$ be a set of generators for $M$.  Then there exists a finite index set $J\subset I$ such that $f(X)\subset \bigoplus_{j\in J} A_j$.  By Proposition \ref{closure of finite direct sums}, there exists a homomorphism $h:M\rightarrow C$ such that $gh(m_i)=f(m_i)$ for all $i\in \{1, \ldots, k\}$.  Since $X$ generates $M$, then $gh=f$, as hoped.
\end{proof}

 
We don't know if the subprojectivity domain of a module is, in general, closed under arbitrary direct sums. However, we have some information regarding when it is closed under arbitrary direct products. If the subprojectivity domain of every module is closed under products, then, by Proposition \ref{basicfact2}, the class of projective modules is closed under products. By \cite[Theorem 3.3]{chase60}, this means that $R$ is a right perfect, left coherent ring. This condition is enough to ensure that subprojectivity domains are closed under products.

\begin{prop}\label{closure under products}
Let $R$ be a ring. The following conditions are equivalent.
\begin{enumerate}
\item $R$ is a right perfect, left coherent ring.
\item The subprojectivity domain of any right $R$-module is closed under arbitrary products.
\end{enumerate}
\end{prop}
\begin{proof}
(2) $\Rightarrow$ (1) follows from the discussion on the preceding paragraph. For (1) $\Rightarrow$ (2), let $M \in \modR$ and let $\{N_{\lambda}\}_{\lambda \in \Lambda}$ be a set of modules in $\subpr{M}$. Let $f = (f_{\lambda})_{\lambda \in \Lambda}: M \rightarrow \prod_{\lambda \in \Lambda} N_{\lambda}$. For every $\lambda \in \Lambda$, be $g_{\lambda}: P_{\lambda} \rightarrow N_{\lambda}$ be an epimorphism with $P_{\lambda}$ projective. By hypothesis, there exists $h_{\lambda}: M \rightarrow P_{\lambda}$ such that $f_{\lambda} = g_{\lambda}h_{\lambda}$. Let $h = (h_{\lambda})_{\lambda \in \Lambda}: M \rightarrow \prod_{\lambda \in \Lambda} P_{\lambda}$, and $g: \prod_{\lambda \in \Lambda} P_{\lambda} \rightarrow \prod{\lambda \in \Lambda} N_{\lambda}$ be defined by $g((x_{\lambda})_{\lambda \in \Lambda}) = (g_\lambda(x_\lambda))_{\lambda \in \Lambda}$. It is routine to check that $g$ is an epimorphism and that $gh = f$. Note that, since $R$ is right perfect and left coherent, $\prod_{\lambda \in \Lambda} P_{\lambda}$ is projective. By Lemma \ref{onlyfrees} (4), $M$ is $\prod_{\lambda \in \Lambda}N_{\lambda}$-subrojective.
\end{proof}

 Similarly, recall that $R$ is said to be right perfect if submodules of projective modules are projective. If every $sp$-portfolio is closed under submodules, then $R$ must be right hereditary. The next proposition tells us that the converse of this statement is also true. 
 
 \begin{prop}\label{Rrighthereditary}
 Let $R$ be a ring. The following conditions are equivalent.
 \begin{enumerate}
 \item $R$ is right hereditary.
 \item The subprojectivity domain of any right $R$-module is closed under submodules
 \end{enumerate}
 \end{prop}
 \begin{proof}
 (2) $\Rightarrow$ (1). If the subprojectivity domain of any right $R$-module is closed under submodules then, by Proposition \ref{basicfact2}, the class of projective modules is closed under submodules. Then, $R$ is right hereditary. For (1) $\Rightarrow$ (2), let $M$ be a right $R$-module, $K \in \subpr{M}$ and $N \leq K$. Let $f: M \rightarrow N$. We can consider $f$ as a morphism from $M$ to $K$ with image in $N$. Let $g: P \rightarrow K$ be an epimorphism with $P$ projective. Then, there exists $h: M \rightarrow P$ such that $gh = f$. Now let $P' = g^{-1}(N) \leq P$. Since $R$ is right hereditary, $P'$ is projective. Note that $h(M) \leq P'$, and $g(P') = N$. By Proposition \ref{onlyfrees} (4), $N \in \subpr{M}$.
 \end{proof}




In general, the subprojectivity domain of a module is not closed with respect to quotients.  Consider for example the $\mathbb{Z}$-modules, $M=\mathbb{Z}/(2)$, $A=\mathbb{Z}$, and $B=2\mathbb{Z}$.  Since $A$ and $B$ are projective, then by Proposition \ref{basicfact2} $A,B \in \subpr{M}$.  But $M=A/B$ is not projective, so by Proposition \ref{basicfact1} $A/B \notin \subpr{M}$. Using similar arguments, note that $\subpr{M}$ is closed under quotients if and only if $M$ is projective.

\section{Subprojectivity domains and torsion-free classes.}

Hereditary pretorsion classes are an important tool in the study of the injective and projective profile of a ring $R$, see \cite{lopez12}. For this reason, it seems reasonable to see if torsion-theoretic notions or techniques may help in the study of $\mathfrak{spP}(R)$. Our next result tells us that it is torsion-free classes thay play a role in the study of this semilattice.

Proposition \ref{hom=0} tells us that, for every module $M$, the torsion-free class generated by $M$, $\mathcal{F}_M$ is contained in $\subpr{M}$. In Example \ref{rationals} we found that, in the category of $\mathbb{Z}$-modules one actually has that $\mathcal{F}_{\mathbb{Q}} = \subpr{\mathbb{Q}}$. Our next goal is to characterize those subprojective portfolios for which this phenomenon happens. First, we give a definition.

\begin{definition}\label{defbasic}
Let $\mathcal{C} \in \mathfrak{spP}(R)$. We say that $\mathcal{C}$ is a {\em basic} $sp$-portfolio if there exists $M \in \modR$ such that $\mathcal{C} = \subpr{M} = \{N \in \modR : \zhom{M}{N} = 0\}$.
\end{definition}

As a quick example, notice that $\modR$ is always a basic $sp$-portfolio, as $\modR = \subpr{0}$. Moreover, if $R$ is a cogenerator for $\modR$ (e.g. a QF-ring) then the only basic subportfolio is $\modR$.

It is clear that if $\mathcal{C}$ is a basic $sp$-portfolio and $M$ is a module as in Definition \ref{defbasic} then $\zhom{M}{P} = 0$ for every projective module $P$. The following proposition tells us that this condition is indeed sufficient for $\subpr{M}$ to be basic.

\begin{prop}\label{basic}
Let $\mathcal{C} \subseteq \modR$ be an $sp$-portfolio. The following conditions are equivalent.
\begin{enumerate}
\item $\mathcal{C}$ is basic.
\item There exists a module $M$ such that $\mathcal{C} = \subpr{M}$ and $\zhom{M}{R} = 0$.
\end{enumerate}
\end{prop}
\begin{proof}
$(1) \Rightarrow (2)$ is clear. We show $(2) \Rightarrow (1)$. Let $N$ be a module such that $\zhom{M}{N} \not= 0$, and let $f: M \rightarrow N$ be a nonzero morphism. Let $p: F \rightarrow N$ be an epimorphism with $F$ free. By (2), $\zhom{M}{F} = 0$, so $f$ cannot be lifted to a morphism $M \rightarrow F$, so $N \not\in \subpr{M}$. Hence, $\subpr{M}$ is basic.
\end{proof}

Note that if $\mathcal{C} = \subpr{M}$ is basic, then it does not follow that $\zhom{M}{R} = 0$. For example, by Proposition, if $M \in \modR$ then $M$ and $M \oplus R$ have the same subprojectivity domain, while it is always the case that $\zhom{M\oplus R}{R} \not= 0$.

As a consequence of Proposition \ref{basic}, we can list the subprojectivity domain of some classes of modules.

\begin{enumerate}
\item $\subpr{\Z_{p^{n}}} = \{M \in \text{Mod-}\mathbb{Z} : M \; \text{does not have elements of order} \; p\}$.
\item $\subpr{\bigoplus_{p \; \text{prime}} \Z_p} = \{M \in \text{Mod-}\mathbb{Z} : \Soc(M) = 0\} = \{M \in \text{Mod-}\mathbb{Z} : t(M) = 0\}$. 
\item $\subpr{\Z_{p^{\infty}}} = \{M \in \text{Mod-}\mathbb{Z} : \Z_{p^{\infty}} \; \text{is not isomorphic to a submodule of} \; M\}$. 
\end{enumerate}

As we have said, for every module $M$, the class $\{N \in \modR : \zhom{M}{N} = 0\}$ is a torsion-free class, that is, it is closed under submodules, extensions and arbitrary direct products. Then, we have the following consequence of Proposition \ref{basic}.

\begin{cor}
Let $M$ be a module such that $\zhom{M}{R} = 0$. Then, $\subpr{M}$ is closed under arbitrary products, submodules, and extensions.
\end{cor}

The module $\Z_{p^{\infty}}$ exhibits an interesting behaviour. It is not projective, but its subprojectivity domain is in some sense large. We formalize this in the following definition.

\begin{definition}
Let $M$ be an $R$-module. We say that $M$ is {\em maximally subprojective if} $\subpr{M}$ is a coatom in $\mathfrak{spP}(R)$.
\end{definition}

\begin{prop}
Let $M$ be a module such that $\subpr{M} = \{K \in \modR : K \; \text{does not have a direct summand isomorphic to} \; M\}$. Then, $M$ is maximally subprojective.
\end{prop}
\begin{proof}
Assume $N$ is a module with $\subpr{M} \subsetneq \subpr{N}$. If $N \cong M \oplus K$, then $\subpr{M} \subsetneq \subpr{N} = \subpr{M}\cap\subpr{K} \subseteq \subpr{M}$, a contradiction. Hence, $N$ does not have direct summands isomorphic to $M$. By our assumptions, $N \in \subpr{M} \subsetneq \subpr{N}$, so $N$ is $N$-subprojective, that is, $N$ is projective. Hence, $M$ is maximally subprojective.
\end{proof}

\begin{cor}
\begin{enumerate} \item $\Z_{p^{\infty}}$ is a maximally subprojective $\Z$-module. \item For any ring $R$, if $S$ is a simple injective nonprojective module over any ring $R$, then $S$ is maximally subprojective. \end{enumerate}
\end{cor}
\begin{proof}
1) is clear from the examples after Proposition \ref{basic}. For 2), if $S$ is a simple injective nonprojective module, then $\zhom{S}{R} = 0$, for otherwise $S$ would be a summand of a projective module. Then, $\subpr{S}$ is basic. Finally, note that those modules $M$ for which $\zhom{S}{M} = 0$ are precisely the modules that do not contain a direct summand isomorphic to $S$. Then, $S$ is maximally subprojective.
\end{proof}

\begin{cor}
Let $R$ be a non-semisimple right $V$-ring. Then, $R$ has maximally subprojective modules.
\end{cor}

Our next result tells us that every $sp$-portfolio is basic if and only if every $sp$-portfolio is a torsion free class and describes precisely for which rings these conditions hold. To state it, we need the following known result.

\begin{prop}[see e.g. \cite{teply}]\label{refsug}
Let $R$ be a ring. The following conditions are equivalent.
\begin{enumerate}
\item $R$ is a right hereditary, right perfect, left coherent ring.
\item $R$ is a semiprimary, right hereditary, left coherent ring.
\end{enumerate}
\end{prop}

We will use this result freely throughout the paper, most notably in Propositions \ref{everyspbasic}, \ref{4.6}, \ref{4.8} and \ref{4.19}. 

\begin{prop}\label{everyspbasic}
Let $R$ be a ring. The following are equivalent.
\begin{enumerate}
\item $R$ is a semiprimary, right hereditary, left coherent ring.
\item Every $sp$-porfolio is basic.
\item Every $sp$-portfolio is a torsion free class.
\end{enumerate}
\end{prop}
\begin{proof}
(2) $\Rightarrow$ (3) is clear. (3) $\Rightarrow$ (1) follows because torsion free classes are closed under arbitrary intersections, so the class of projective modules is a torsion free class, cf. Proposition \ref{basicfact2}. Finally, for (1) $\Rightarrow$ (2), let $M$ be a right $R$-module. Since $R$ is right hereditary, right perfect, left coherent, the class of projective modules is closed under direct products and submodules. Let $N = \bigcap\{\Ker(f) : f: M \rightarrow P, \; \text{and} \; P \; \text{is projective}\}$. $M/N$ is projective and $N$ is the smallest submodule of $M$ that yields a projective quotient. Now, $M \cong M/N \oplus K$. If $\zhom{K}{R} \not= 0$ then $K$ has a projective quotient and we can find a submodule of $M$ properly contained in $N$ that yields a projective quotient of $M$, a contradiction. Hence, $\zhom{K}{R} = 0$, so $\subpr{K}$ is basic, and $\subpr{M} = \subpr{M/N}\cap\subpr{K} = \subpr{K}$.
\end{proof}

Note that if we ignore (2) in Proposition \ref{everyspbasic}, the equivalence (1) $\Leftrightarrow$ (3) can be easily obtained from propositions \ref{closure under products} and \ref{Rrighthereditary}. \\

The use of torsion-theoretic techniques can also be applied to study the notion of subinjectivity, as defined in \cite{pinar11}. 

\begin{definition}
Let $\mathcal{I} \subseteq \modR$. We say that $\mathcal{I}$ is a {\em subinjective-portfolio}, or {\em $si$-portfolio} for short, if there exists $M \in \modR$ such that $\mathcal{I} = \underline{\mathfrak{In}}^{-1}(M)$. The class $\mathfrak{siP}(R) := \{\mathcal{I} \subseteq \modR : I \; \text{is an si-portfolio}\}$ will be called the {\em subinjective profile}, or {\em $si$-profile} of $R$. By \cite[Proposition 2.4(1)]{pinar11}, $\mathfrak{siP}(R)$ is a semilattice with biggest element.
\end{definition}

Analog to Proposition \ref{hom=0}, we have the following result, that tells us that, for every module $M$, the torsion class cogenerated by $M$ $\mathcal{T}^{M}$ is contained in $\underline{\mathfrak{In}}^{-1}(M)$. 

\begin{prop}\label{hom=0injective}
Let $M, N \in \modR$. If $\zhom{N}{M} = 0$, then $N \in \underline{\mathfrak{In}}^{-1}(M)$.
\end{prop}

Motivated by Proposition \ref{hom=0injective}, we have the following definition.

 \begin{definition}
We say that an $si$-portfolio $\mathcal{I}$ is {\em basic} if there exists a module $M$ for which $\mathcal{I} = \subinj{M} = \{N \in \modR : \zhom{N}{M} = 0\}$. 
\end{definition}

\begin{prop}
Let $\mathcal{I} \subseteq \modR$ be an $si$-portfolio. The following conditions are equivalent.
\begin{enumerate}
\item $\mathcal{I}$ is basic.
\item There exists $M \in \modR$ such that $\mathcal{I} = \subinj{M}$ and $\zhom{E}{M} = 0$ for every injective module $E$.
\end{enumerate}

If $R$ is right noetherian, this happens if and only if $\zhom{\mathcal{E}}{M} = 0$, where $\mathcal{E} = \bigoplus\{E: E \; \text{is an indecomposable injective}\}$.
\end{prop}
\begin{proof}
(1) $\Rightarrow$ (2) is clear from \cite[Proposition 2.3]{pinar11}. Now assume (2), and let $N \in \modR$ be such that $\zhom{N}{M}       \not= 0$. Then, there exists a nonzero $f: N \rightarrow M$ that cannot be extended to a morphism $\overline{f}: E(N) \rightarrow M$. Hence, $N \not\in \subinj{M}$. The last assertion is clear.
\end{proof}

It is again worth noticing that if $\mathcal{I} = \subinj{M}$ is basic, it does not follow that $\zhom{E}{M} = 0$ for every injective module $E$. For example, if $E_0$ is any injective module, it follows from \cite[Proposition 2.4 (2)]{pinar11} that $M$ and $M \oplus E_0$ have the same subinjectivity domain, while it is always the case that $\zhom{E_0}{E_0 \oplus M} \not= 0$.

\begin{example}\label{subinjofZ}
The subinjectivity domain of $\mathbb{Z}$ is $\subinj{\Z} = \{N \in \Z\text{-}\Mod : \zhom{N}{\Z} = 0\} = \{N \in \Z\text{-}\Mod : \Z \; \text{is not isomorphic to a direct summand of} \; N\}$.
\end{example}

Clearly, $\Z$ is not an injective (= divisible) $\Z$-module. However, it is proved in \cite{lopez12} that $\mathfrak{In}^{-1}(\Z)$ is a coatom in the injective profile of $\Z$. Interestingly enough, $\subinj{\Z}$ is also a coatom in the subinjective profile of $\Z$.

\begin{definition}
Let $M \in \modR$. We say that $M$ is {\em maximally subinjective} if $\subinj{M}$ is a coatom in the subinjective profile of $R$.
\end{definition}

\begin{prop}
Let $M$ be an $R$-module such that $\subinj{M} = \{N \in \modR : M \; \text{is not isomorphic to a direct summand of} \; N\}$. Then, $M$ is maximally subinjective.
\end{prop}
\begin{proof}
Let $K$ be a module such that $\subinj{M} \subsetneq \subinj{K}$. Note that, if $K \cong L \oplus M$, then $\subinj{K} = \subinj{L}\cap\subinj{M} \subseteq \subinj{M}$, a contradiction. Hence, $M$ is not isomorphic to a direct summand of $K$, so $K \in \subinj{M} \subsetneq \subinj{K}$, that is $K$ is $K$-subinjective. Therefore, $K$ is injective and $M$ is maximally subinjective.
\end{proof}

Again, note that $\modR = \subinj{0}$ is always a basic $si$-portfolio. If $\modR$ has an injective generator (e.g. a right self-injective ring) then $\modR$ is the only basic $si$-portfolio. The next proposition tackles the extreme opposite case, that is, when every $si$-portfolio is basic. 
 
\begin{prop}
Let $R$ be a ring. The following conditions are equivalent.
\begin{enumerate}
\item $R$ is a right hereditary, right noetherian ring.
\item Every $si$-portfolio is basic.
\item Every $si$-portfolio is a torsion class.
\end{enumerate}
\end{prop}
\begin{proof}
(2) $\Rightarrow$ (3) $\Rightarrow$ (1) is clear. For (1) $\Rightarrow$ (2), assume $R$ is right hereditary, right noetherian and let $M \in \modR$. Then, $M = d(M) \oplus r(M)$, where $d(M)$ is the divisible part of $M$ and $r(M)$ its reduced part. Since $r(M)$ does not have injective submodules, $\zhom{E}{r(M)} = 0$ for every injective module $E$. Then, $\subinj{M}$ is basic and $\subinj{M} = \subinj{d(M)}\cap\subinj{r(M)} = \subinj{r(M)$}.
\end{proof}

Now, assume $R$ is a right hereditary, right noetherian ring. Then, the class $\mathfrak{E}$ of injective modules is a torsion class. Let $\mathfrak{F}$ denote its corresponding torsion free class. Then, $M \in \mathfrak{F}$ if and only if $\zhom{E}{M} = 0$ for every $E \in \mathfrak{E}$, that is, if and only if $\subinj{M} = \{N \in \modR : \zhom{N}{M} = 0\}$. Then, a module $K \in \mathfrak{F}$ is indigent if and only if $\{N \in \modR : \zhom{N}{K} = 0\} = \mathfrak{E}$, if and only if $K$ is a cogenerator of the torsion theory $(\mathfrak{E}, \mathfrak{F})$. With these observations, we have proved the following result.

\begin{prop}\label{hereditarynoetherianring}
Let $R$ be a right hereditary right noetherian ring, and let $\mathbb{T} = (\mathfrak{E}, \mathfrak{F})$ be the torsion theory where $\mathfrak{E}$ is the class of injective modules. Then, a module $M$ is $si$-poor if and only if $M/d(M)$ is a cogenerator of $\mathbb{T}$. In particular, $R$ has an $si$-poor module if and only if $\mathbb{T}$ can be cogenerated by a single module.
\end{prop}

As an application of Proposition \ref{hereditarynoetherianring}, we show that $si$-poor modules exist over a hereditary finite dimensional algebra over an algebraically closed field $k = \overline{k}$. Recall that, over such an algebra $A$, in mod-$A$ we have the Auslander-Reiten translate $\tau$, where $\tau M$ is the $k$-dual of the transpose of $M$. The Auslander-Reiten translate has several interesting properties, a number of which can be found in \cite[Chapter IV]{assem}.

\begin{cor}
Let $A$ be a hereditary finite dimensional algebra over an algebraically closed field $k$, and let $\mathcal{E}$ be the direct sum of indecomposable injectives. Then, $\tau\mathcal{E}$ is an $si$-poor module, where $\tau$ denotes the Auslander-Reiten translate.
\end{cor}
\begin{proof}
Since $\mathcal{E}$ is injective, $\text{Ext}^{1}_{A}(\mathcal{E}, \mathcal{E}) = 0$. Since $A$ is hereditary, $\text{pr.dim}(\mathcal{E}) \leq 1$. The number of non-isomorphic indecomposable summands of $\mathcal{E}$ equals the rank of the Grothendieck group $K_0(A)$. Then, $\mathcal{E}$ is a tilting $A$-module, see \cite[Corollary VI.4.4]{assem}. It follows that $(\text{Gen}(\mathcal{E}),\text{Cogen}(\tau\mathcal{E}))$ is a torsion theory \cite[Theorem VI.2.5]{assem}. Note that $\text{Gen}(\mathcal{E})$ is the class of injective modules. Hence, $\tau\mathcal{E}$ is an indigent module.
\end{proof}

\begin{example}
Let $A$ be the path algebra of the quiver

$$
\xymatrix{1 \ar[r] & 2 & 3 \ar[l] \ar[r] & 4}
$$

The indecomposable projectives are $P(1) = K \rightarrow K \leftarrow 0 \rightarrow 0$, $P(2) = 0 \rightarrow K \leftarrow 0 \rightarrow 0$, $P(3) = 0 \rightarrow K \leftarrow K \rightarrow K$, and $P(4) = 0 \rightarrow 0 \leftarrow 0 \rightarrow K$; and the indecomposable injectives $I(1) = K \rightarrow 0 \leftarrow 0 \rightarrow 0$, $I(2) = K \rightarrow K \leftarrow K \rightarrow 0$, $I(3) = 0 \rightarrow 0 \leftarrow K \rightarrow 0$ and $P(4) = 0 \rightarrow 0 \leftarrow K \rightarrow K$. The Auslander-Reiten quiver of $A$, $\Gamma(A)$, is

$$
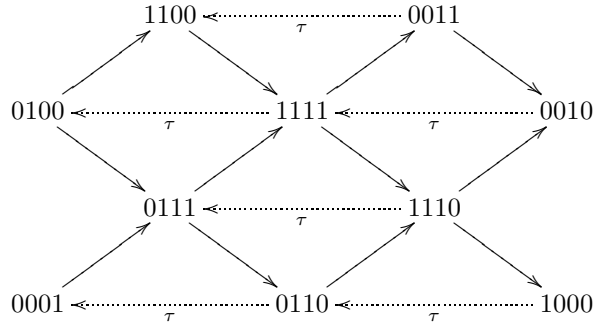
\xymatrix{ & 1100 \ar[dr] & & 0011 \ar[dr] \ar@{.>}[ll]^{\tau} & \\ 
0100 \ar[ur] \ar[dr] & & 1111 \ar[ur] \ar[dr] \ar@{.>}[ll]^{\tau} & & 0010 \ar@{.>}[ll]^{\tau} \\ 
& 0111 \ar[ur] \ar[dr] & & 1110 \ar[ur] \ar[dr] \ar@{.>}[ll]^{\tau} & \\ 
0001 \ar[ur] & & 0110 \ar@{.>}[ll]^{\tau} \ar[ur] & & 1000 \ar@{.>}[ll]^{\tau}}
$$

where we represent each module by its dimension vector. Now, the sum of indecomposable injectives is $E = 1000 \oplus 1110 \oplus 0010 \oplus 0011$, so $\tau E = 0110 \oplus 0111 \oplus 1111 \oplus 1100$ is $si$-poor.
\end{example}



\section{Bounds in the subprojectivity domain of a module.}
\label{p-indigentSection}

In this section, we investigate both upper and lower bounds one may impose on the subprojectivity domain of a module. Recall that a module is said to be $sp$-poor, or $p$-indigent, if its subprojectivity domain consists precisely of the projective modules. We will show next that this condition may be softened with equivalent results.

\begin{prop}
Consider the following conditions on a module $M$.
\begin{enumerate}
\item $\subpr{M} = \{P \in \modR : P \; \text{is projective}\}$.
\item $\subpr{M} \subseteq \{P \in \modR : P \; \text{is quasi-projective}\}$.
\item $\subpr{M} \subseteq \{P \in \modR : P \; \text{is discrete}\}$. 
\item $\subpr{M} \subseteq \{P \in \modR : P \; \text{is quasi-discrete}\}$.
\end{enumerate}
Then, (1) and (2) are equivalent, and the four conditions are equivalent if $R$ is a right perfect ring.
\end{prop}
\begin{proof}
(1) $\Rightarrow$ (2) is clear. Assume (2). Let $K \in \subpr{M}$, and let $P$ be a projective module that covers $K$. By Proposition \ref{closure of finite direct sums}, $K \oplus P \in \subpr{M}$. Then, $K \oplus P$ is quasi-projective, so $K$ is $P$-projective and hence projective. If $R$ is right perfect, then every quasi-projective module is discrete (cf. \cite[Theorem 4.41]{mohamed90}), so we have (1) $\Rightarrow$ (2) $\Rightarrow$ (3) $\Rightarrow$ (4). To show (4) $\Rightarrow$ (1), assume (4) and let $K \in \subpr{M}$. Then, $K$ is quasi-discrete, so by \cite[Theorem 4.15]{mohamed90} there exists a decomposition $K = \bigoplus_{i \in I} K_i$, where each $K_i$ is a quasi-discrete hollow module. We show that each $K_i$ is projective. Indeed, let $P$ be the projective cover of $K_i$. Since $R$ is perfect, $P$ is quasi-discrete, so $P = \bigoplus_{j \in J} P_j$, where each $P_j$ is a projective hollow module. By Proposition \ref{sums in subprojectivity domain}, each $P_j$ is in $\subpr{M}$. Then, $P_j \oplus K_i \in \subpr{M}$, so $P_j \oplus K_i$ is quasi-discrete. Then, by \cite[Theorem 4.48]{mohamed90} $K_i$ is $P_j$-projective. Since $R$ is right perfect, projectivity domains are closed under arbitrary direct sums (\cite[Exercise 7.16]{anderson92}), which implies that $K_i$ is $P$-projective. Then, $K_i$ is projective.
\end{proof}

Note that if $R$ is not right perfect then the implication (2) $\Rightarrow$ (3) does not hold as, by \cite[Theorem 4.41]{mohamed90}, a ring is right perfect if and only if every quasi-projective right module is discrete.


\begin{definition}
A ring $R$ is called {\bf right manageable} if there exist a set $S$ of  non-projective right $R$-modules such that for every non-projective $R$-module $M$, there exists $A \in S$ such that $A$ is isomorphic to a direct summand of $M$. For convenience, we refer to the set  $S$ as the manageable set associated with $R$.
\end{definition}

Recall that a ring $R$ is said to be $\Sigma$-cyclic if every right $R$-module is a direct sum of cyclic modules. See \cite[Chapter 25]{faith76}. 
\begin{example}
If $R$ is a right $\Sigma$-cyclic ring, then $R$ is  right manageable. In particular, an artinian serial ring is both right and left manageable.
\end{example}

\begin{prop}\label{manageable have pindigents}
Every manageable ring $R$ has an $sp$-poor module.
\end{prop}

\begin{proof}
Let $S$ be the manageable set of modules associated with $R$.  Let $X=\bigoplus_{A \in S}A$.  We claim that $X$ is $sp$-poor.  To see this, let $B \in \subpr{X}$. If $B$ is not projective, then there exists $C \in S$ such that $C$ is isomorphic to a direct summand of $B$. By Proposition \ref{sums in subprojectivity domain} $C \in \subpr{X}$. By Proposition \ref{subprojectivity domains of sums}, $C \in \subpr{C}$ and by Proposition \ref{basicfact1}, $C$ is projective, a contradiction. Then, $B$ is projective and $X$ is $sp$-poor.
\end{proof}

If $R$ is an artinian chain ring then Proposition \ref{manageable have pindigents} implies that the direct sum of nonprojective cyclic right $R$-modules is $sp$-poor. The next proposition tells us that, in fact, for such a ring every non-projective right $R$-module is $sp$-poor.  This is an interesting discovery giving us a glance into the phenomenon of a ring $R$ having no subprojective middle class.  It should be noted that artinian chain rings also fail to have a subinjective middle class \cite{pinar11}. The study of rings without a projective or injective middle class has been undertaken in \cite{lopez10}, \cite{lopez11}, \cite{lopez12} and \cite{holston11}.
\begin{prop}
If $R$ is an artinian chain ring, then every non-projective module is $sp$-poor.
\end{prop}

\begin{proof}
Since $R$ is an artinian chain ring, every $R$-module is a direct sum of cyclic uniserial modules.  Consequently, it suffices to consider cyclic modules by Proposition \ref{sums in subprojectivity domain} and Proposition \ref{subprojectivity domains of sums}.  Because $R$ is an artinian chain ring, the ideals of $R$ are zero or the powers $J(R)^n$ of $J(R)$, the Jacobson radical of $R$.  Moreover, if $p \in J(R)$ but $p \notin J(R)^2$, then $J(R)^n=p^nR$ for every $n \geq 0$.  Hence, we have the finite chain for some positive integer $n$:
\[R \supset pR \supset p^2R \supset \ldots \supset p^nR=0.\]
Therefore, it is enough to show that $p^kR$ is $sp$-poor for every positive integer $k$.

Let $A=p^kR$, where $k \neq 0$ and let $g:R \rightarrow p^kR$ be the quotient map.  If $k > m$, then let $f:A \rightarrow p^mR$ be the inclusion map.  Assume there exists $h:A \rightarrow R$ such that $gh=f$.  Since $R$ is a chain ring either $\Ker g \subset \Ima h$ or $\Ima h \subset \Ker g$.  If $\Ima h \subset \Ker g$ then $gh=0$, a contradiction.  Hence, $\Ker g \subset \Ima h$.  But since $g$ is not monic, then $\Ker g \neq 0$.  Hence, there is a nonzero element $x \in A$ such that $0=gh(x)=f(x)$, a contradiction.  Thus, $p^mR \notin \subpr{A}$.

If $k < m$, then consider the homomorphism $f:A \rightarrow p^mR$, where $f(p^k)=p^m$.  Assume there exists $h:A \rightarrow R$ such that $gh=f$.  But then $p^m=f(p^k)=gh(p^k)=g(1)p^m \in p^{(2m)}R$, a contradiction.  Thus, $p^mR \notin \subpr{A}$.

Finally, we note that by Proposition \ref{basicfact1}, $p^kR \notin \subpr{p^kR}$, as $p^kR$ is not projective.
\end{proof}

\begin{example}
If $p$ is a prime, then $\mathbb{Z}/p^i\mathbb{Z}$ is a $p$-indigent $\left( \mathbb{Z}/p^k\mathbb{Z}\right)$-module for every $i < k$.
\end{example}

Now we investigate the existance of $sp$-poor modules over a semiprimary, right hereditary, left coherent ring. We choose this class of rings because it is precisely when the projective modules form a torsion-free class $\mathfrak{P}$, see Proposition \ref{refsug}. 

\begin{prop}\label{4.6}
Let $R$ be a semiprimary, right hereditary, left coherent ring. Let $\mathfrak{P}$ be the torsion-free class consisting of the projective $R$-modules, and let $\mathfrak{T}$ be the corresponding torsion class. Then, $R$ has an $sp$-poor module if and only if there exists a module $M$ that generates $(\mathfrak{T}, \mathfrak{P}).$
\end{prop}
\begin{proof}
Assume there exists a module $M$ that generates $(\mathfrak{T}, \mathfrak{P})$. Then, since $(\mathfrak{T}, \mathfrak{P})$ is a torsion theory, the class $\{N: \zhom{M}{N} = 0\} = \mathfrak{P}$. Then, $M$ is $sp$-poor. Now, assume $M$ is an $sp$-poor module. Let $M'$ be the smallest module that yields a projective quotient, so $M \cong M/M' \oplus N$ and $\subpr{M} = \subpr{N}$, with $N \in \mathfrak{T}$. Now $\subpr{N}$ is basic, and $\subpr{N} = \mathfrak{P}$. Therefore, $N$ is a generator of $(\mathfrak{T}, \mathfrak{P})$.
\end{proof}

Now we investigate the existence of an $sp$-poor $\Z$-module. Note that $\Z$ is not perfect, so we cannot apply the preceding proposition. In fact, we have the following.

\begin{prop}\label{4.8}
Let $R$ be a ring which is not semiprimary, or not right hereditary, or not left coherent. If there exists an $sp$-poor module $M$, then $\zhom{M}{R} \not= 0$.
\end{prop}
\begin{proof}
If $\zhom{M}{R} = 0$, then $\subpr{M}$ is a torsion-free class. But this cannot happen, as the class of projective modules is either not closed under submodules or not closed under arbitrary direct products. Then, $\zhom{M}{R} \not= 0$.
\end{proof}

\begin{cor}
If $M$ is an $sp$-poor $\Z$-module, then $\Hom_{\Z}(M,\Z) \neq 0$ and, consequently, $\Hom_{\Z}(M, N) \neq 0$ for every abelian group $N$.
\end{cor}

Since $\Z$ is a principal ideal domain, the last corollary tells us that if $M$ is an $sp$-poor $\Z$-module, then $M \cong \Z \oplus N$, and $\subpr{M} = \subpr{\Z}\cap\subpr{N} = \subpr{N}$, so $N$ is also an $sp$-poor $\Z$-module. Iterating the process and taking a direct limit, we have the following result.

\begin{cor}\label{pindigentverybig}
Let $M$ be a $p$-indigent $\Z$-module. Then, there exist a submodule $N \leq M$ such that $N \cong \Z^{(\mathbb{N})}$.
\end{cor}

Our next goal is to show that the $\Z$-modules $T=\prod_i (\Z/p_i\Z)$ and $S=\left(\prod_i (\Z/p_i\Z\right)/\left(\bigoplus_i \Z/p_i\Z)\right)$ are not $sp$-poor, where $p_1<p_2<\ldots$ are the rational primes in increasing order.  To do so, we will need the following result, which can be found in \cite{fuchs}. For each $i \in \mathbb{N}$, let $e_i \in \baer$ be the standard unit vectors in $\baer$, that is, $e_i(j) = \delta_{ij}$, the Kronecker delta.

\begin{prop}\label{Fuchs result}
Every homomorphism $f:\baer \rightarrow \Z$ is completely determined by its action on $\cat$.  In particular, if $f(e_i)=0$ for all $i$, then $f=0$.
\end{prop}

\begin{prop}\label{lion not p-indigent}
$\zhom{T}{\Z}=0$.
\end{prop}

\begin{proof}
Let $f \in \zhom{T}{\Z}$.  Define $P = \Z^{\mathbb{N}}$ and $g:P \rightarrow T$ by $[g(\alpha)](i):=\alpha(i)+p_i\Z \in \Z/p_i\Z$; that is
\[(\alpha_1, \alpha_2, \ldots) \mapsto (\alpha_1+p_1\Z, \alpha_2+p_2\Z, \ldots).\]
Then $g$ is epic and $fg \in \zhom{P}{\Z}$. 

Fix $k \in \mathbb{N}$.  Then $g(p_ke_k)=0$.  Hence $0=(fg)(p_ke_k)=p_k(fg)(e_k) \in \Z$.  Hence $(fg)(e_k)=0$.

So what we have shown is that $(fg)(e_i)=0$ for all $i$, and so by the last statement of Proposition \ref{Fuchs result}, $fg=0$.  Since $g$ is epic, it follows that $f=0$, which concludes the proof.
\end{proof}

\begin{prop}\label{tiger not p-indigent}
$\zhom{S}{\Z}=0$.
\end{prop}

\begin{proof}
Let $f \in \zhom{S}{\Z}$.  Let $h:T \rightarrow S$ be the epic mapping each element to its equivalence class.  Then $fh \in \zhom{T}{\Z}$.  By \ref{lion not p-indigent}, $fh=0$.  Since $h$ is epic, it follows that $f=0$, which concludes the proof.
\end{proof}

From Propositions \ref{lion not p-indigent} and \ref{tiger not p-indigent}, we conclude that both $S$ and $T$ are not $sp$-poor $\Z$-modules. In view of Corollary \ref{pindigentverybig}, another natural candidate for an $sp$-poor $\Z$-module is the Baer-Specker group $\Z^{\mathbb{N}}$. However, we don't know if this is the case. Also note that, by \cite[Lemma 2.8']{lam99}, the group $\Z^{\mathbb{N}}/\Z^{(\mathbb{N})}$ is also not $p$-indigent.

Now we consider a lower bound on $\subpr{M}$ which is  inspired by \cite{amin}, where they define a module $M$ to be strongly soc-injective if, for every $N \in \modR$, every morphism $f: \Soc(N) \rightarrow M$ can be extended to a morphism $\overline{f}: N \rightarrow M$. It is not hard to see that the requirement of $M$ to be strongly soc-injective is equivalent to saying that $\ssmod\text{-}R \subseteq \subinj{M}$.

\begin{definition}
Let $M \in \modR$. We say that $M$ is strongly soc-projective if $\ssmod\text{-}R \subseteq \subpr{M}$.
\end{definition}

Of course, a strongly soc-projective module need not be projective, as the $\Z$-module $\mathbb{Q}_{\Z}$ shows. However, in the category of abelian groups, a finitely generated strongly-soc projective module is projective.

\begin{prop}\label{fg abelian groups}
Let $M$ be a finitely generated abelian group such that every semisimple module is in $\subpr{M}$. Then, $M$ is projective.
\end{prop}
\begin{proof}
By the Fundamental Theorem of Finitely Generated Abelian groups, $M \cong \mathbb{Z}^n \oplus \mathbb{Z}_{p_1^{\alpha_1}} \oplus \dots \oplus \mathbb{Z}_{p_n^{\alpha_n}}$. If $\alpha_i \not= 0$ for some $i$, then $\mathbb{Z}_{p_i} \not\in \subpr{M}$, a contradiction. Then, $M \cong \mathbb{Z}^n$ is projective.
\end{proof}

If $R$ is a semiperfect ring, then every simple module has a projective cover, which has to be a local module, that is, with only one maximal submodule. In this case, we have the following proposition.

\begin{prop}\label{1}
Let $R$ be a semiperfect ring and let $M$ be a right $R$-module. Assume $S$ is a simple module such that $S \in \subpr{M}$. Then, either $\zhom{M}{S} = 0$ or $M = P(S) \oplus K$, where $P(S)$ stands for the projective cover of $S$.
\end{prop}
\begin{proof}
Assume $\zhom{M}{S} \not= 0$, and let $f: M \rightarrow S$ be a nonzero morphism. By hypothesis, we can lift this morphism to a morphism $\overline{f}: M \rightarrow P(S)$. Now, $\overline{f}(M)$ cannot be contained in the unique maximal ideal of $P(S)$, which is the kernel of the epimorphism $P(S) \rightarrow S$. Then, $\overline{f}(M) = P(S)$ and, by the projectivity of $P(S)$, we conclude that $M = P(S) \oplus K$.
\end{proof}

\begin{cor}\label{2}
Let $R$ be a semiperfect ring and let $M$ be a finitely generated right $R$-module of finite uniform dimension that is subprojective with respect to every simple (equivalently, with respect to every semisimple) module. Then, $M$ is projective.
\end{cor}
\begin{proof}
Iterating the process of Proposition \ref{1}, and using the fact that $M$ has finite uniform dimension, we have that $M = P \oplus K$, with $P$ projective and $K$ has no nonzero morphisms to a simple module. If $K \not= 0$ then $K$ has maximal submodules, a contradiction. Hence, $M = P$
\end{proof}

Note that the conclusion of Corollary \ref{2} may hold even if $R$ is not semiperfect. For example, by Proposition \ref{fg abelian groups}, it holds for the ring of integers $\Z$. \\

If we assume that $R$ is right perfect, we can drop the finitely generated assumption in Corollary \ref{2}, as the existence of a maximal submodule of $K$ is guaranteed by the conditions on $R$. Then, we have the following.

\begin{cor}\label{2.1}
Let $R$ be a right perfect ring and let $M$ be a right $R$-module of finite uniform dimension that is subprojective with respect to every semisimple module. Then, $M$ is projective.
\end{cor}

If, moreover, we assume that our ring is semiprimary, right hereditary, left coherent, we can also remove the finite uniform dimension assumption in Corollary \ref{2.1}. 

\begin{prop}\label{4.19}
Let $R$ be a semiprimary, right hereditary, left coherent ring. Then, a right module $M$ is projective if and only if it is strongly soc-projective.
\end{prop}
\begin{proof}
As we've seen, in this case every module has a smallest module that yields a projective quotient. Then, we can decompose every module $M$ as $M \cong P \oplus X$, with $P$ projective and $X$ a module without projective quotients. Moreover, since $R$ is right perfect, every nonzero module has maximal submodules, cf. \cite{bass60}. Then, if $X \not= 0$ there exists a simple module $S$ such that $\zhom{X}{S} \not= 0$. Since $\subpr{M} = \subpr{X}$, $S \in \subpr{X}$. This implies, by Proposition \ref{1}, that $P(S)$ is a direct summand of $X$, a contradiction. Hence, $X = 0$ and $M \cong P$ is projective.
\end{proof}
\\

\textbf{Acknowledgment.} The authors thank the anonymous referee for many valuable comments and suggestions that allowed us to improve the exposition of this paper.

\end{document}